\documentclass[a4paper,10pt]{article}
\usepackage{amsmath, amsthm, amssymb}
\usepackage{setspace}
\usepackage{tikz}
\usetikzlibrary{matrix,arrows}
\usepackage{pb-diagram}
\newtheorem{theorem}{Theorem}[section]
\newtheorem{corollary}[theorem]{Corollary}

\newtheorem{lemma}[theorem]{Lemma}

\def\im{\hbox{\rm im}}
 at 12 truept

\def\im{\hbox{\rm im}}
\def\mod{\hbox{\rm mod}}
\def\tensor{\otimes}

\begin{document}

\title{On free products and amalgams of pomonoids}
\author{Bana Al Subaiei and James Renshaw\\\small Department of Mathematics\\
\small University of Southampton\\
\small Southampton, SO17 1BJ\\
\small England\\
\small Email: banauk@hotmail.com\\
\small j.h.renshaw@maths.soton.ac.uk}
\date{December 2013}
\maketitle

\begin{abstract}
The study of amalgamation in the category of partially ordered monoids was initiated by Fakhuruddin in the 1980s. In  1986 he proved that, in the category of commutative pomonoids, every absolutely flat commutative pomonoid is a weak amalgmation base and every commutative pogroup is a strong amalgamation base. Some twenty years later, Bulman-Fleming  and Sohail in 2011 extended this work to what they referred to as pomonoid amalgams. In particular they proved that pogroups are poamalgmation bases in the category of pomonoids. Sohail, also in 2011, proved that absolutely poflat commutative pomonoids are poamalgmation bases in the category of commutative pomonoids.
In the present paper we extend the work on pomonoid amalgams by generalising the work of Renshaw on amalgams of monoids and extension properties of acts over monoids.
\end{abstract}

\medskip



\section{Introduction and Preliminaries}
For background material on semigroups, monoids and acts over monoids we refer the reader to \cite{howie-1995} and for that relating to $S-$posets we refer to \cite{bf-2005} and \cite{bf-2006}.
We start with a brief resum\'e of the category of $S-$posets, where $S$ is a pomonoid, and introduce some basic results on tensor products, direct systems and pushouts, many of which are needed for later sections. Section 2 introduces the important concept of a {\em free extension} and we show how the amalgamated free product of an amalgam of pomonoids can be viewed as a direct limit of posets constructed from these extensions. This work is very similar to the situation in the unordered case (see~\cite{renshaw-1986}) and therefore some of the technical details have been omitted for brevity as they follow a similar argument. However there are a number of differences in the ordered case and we have illuminated these where necessary. An ordered version of the extension property is introduced and it is shown that {\em weak amalgamation pairs} have this extension property. Section 3 concerns a number of ordered versions of the {\em unitary property}. Unitary submonoids were first shown to be strongly connected to amalgamation by Howie in his pioneering work~\cite{howie-1962}. We introduce the concept of {\em strongly pounitary} submonoids, which we feel is a more natural analogue of unitary submonoids, particularly with respect to amalgamation. We culminate this section with an ordered version of Howie's original result. Finally, in section 4, we briefly consider the situation for commutative pomonoids.

\bigskip

A monoid $S$ is said to be a {\em partially ordered monoid} or a {\em pomonoid} if $S$ is endowed with a partial order $\le$ which is compatible with the binary operation on $S$ in the following manner
$$
\forall s, t, u \in S, t \le u \Rightarrow st \le su  \text{ and }  ts \le us.
$$
A map $f$: $X \to Y$, where $X$ and $Y$ are posets, is said to be {\em monotone} if for all $x, y \in X, x \le y \Rightarrow f(x) \le f(y)$, whereas it is said to be an {\em order embedding}  if for all $x, y \in X, x \le y \Leftrightarrow f(x) \le f(y)$.  Clearly order embeddings are one-to-one and monotone. A map $f$ is said to be an {\em order isomorphism} if it is a surjective order embedding.
It is worth noting that in the category of posets, to demonstrate that a map $f:X\to Y$ is well defined, it suffices to show that it satisfies the {\em monotonic property} $f(x)\le f(y)$ whenever $x\le y$. We shall make frequent use of this without further reference.

If $S$ is a pomonoid and $A$ is a non empty poset, then $A$ is called a {\em right $S-$poset} if $A$ is a right $S-$act and the action is monotonic in each of the variables. That is to say

\begin{enumerate}
\item $a1$ = $a$  and  $a(st)$ = $(as)t$  for all $s,t\in S,a\in A$;
\item if $a\le b\in A, s\in S$ then $as\le bs$;
\item if $a\in A, s\le t\in S$ then $as\le at$.
\end{enumerate}
{\em Left $S-$posets} are defined dually. If $A$ is both a left $S-$poset and a right $T-$poset for pomonoids $S$ and $T$, and in addition $(sa)t=s(at)$ for all $s\in S, a\in A, t\in T$ then we call $A$ an {\em $(S,T)-$poset}. If $A$ and $B$ are $S-$posets then the map $f: A \to B$ is said to be an {\em $S-$poset morphism} when $f$ is both monotonic and a morphism of $S-$acts. In the category of $S-$posets the monomorphisms are exactly the one-to-one $S-$poset morphisms and the epimorphisms are exactly the onto $S-$poset morphisms \cite{bf-2006}. 

\smallskip

As in~\cite{bf-2006}, a {\em congruence} $\theta$ on an $S-$poset $A$ is an $S-$act congruence with the additional property that $A/\theta$ can be endowed with a suitable partial order so that $A/\theta$ is an $S-$poset and the canconical map $\theta^\natural: A\to A/\theta$ is an $S-$poset morphism.
It can be shown that an $S-$act congruence $\rho$ on an $S-$poset $A$ is an $S-$poset congruence if and only if 
 $a \; \rho \; b$ whenever $a \le_{\rho} b  \le_{\rho} a$
 where $a \le_{\rho} b$ is defined by
$$
a\le_\rho b\text{ if and only if there exist } n\ge1\text{ and }a_1,a_1',\ldots a_n, a_n'\in A,
$$
$$
 a \le a_{1}\; \rho\; a'_{1}\le a_{2}\; \rho\; a'_{2} \dots \le a_{n}\; \rho\; a'_{n} \le b.
$$
If $R$ is a binary relation on a right $S-$poset $A$ then the {\em right $S-$poset congruence $\nu(R)$ induced on $A$ by $R$} is defined as
$$
a \; \nu(R) \; b \text{ if and only if } a \le_{\alpha(R)} b \le_{\alpha(R)} a
$$
where the relation $\alpha(R)$ is defined by $a  \; \alpha(R)  \; b$ if and only if either $a=b$ or there exists $n\ge 1$ and $(x_i,x_i')\in R, s_i\in S$ for $i=1,\ldots, n$ such that
$$
a = x_{1}s_{1},  x'_{1}s_{1} = x_{2}s_{2}, \dots , x'_{n}s_{n} = b.
$$
The relation $\alpha(R)$ is reflexive and transitive and also preserves the action of $S$. The order relation on $A/\nu(R)$ is defined as $ a \nu(R) \le b \nu(R)$ if and only if $a \le_{\alpha(R)} b.$ 
It is easy to show that for any  binary relation $R$ on $A$, if $(x, x') \in R$, then $(x, x') \in \alpha(R)$, and thus $x \le_{\alpha(R)} x'$.

The right $S-$poset congruence $\theta(R)=\nu(R\cup R^{-1})$ is the {\it congruence generated}  (in the usual sense) on $A$ by $R$.

\bigskip

In the 1980s, Fakhruddin (\cite{fakhruddin-1986}) initiated the study of pomonoid amalgams, which has recently been extended by Bulman-Fleming and Sohail (\cite{bf-2010}).
A {\em pomonoid amalgam} ${\cal A} = [U; S_{i}; \varphi_{i}]$ consists of a pomonoid $U$, called the {\em core}, a family   $\lbrace S_{i}: i\in I\rbrace$ of pomonoids, and a family $\lbrace\varphi_{i}: i\in I\rbrace$ of pomonoid order embeddings, $\varphi_{i}: U\to S_{i}$. Pomonoid order embeddings are monoid homomorphisms that are order-embeddings of the underlying posets. The amalgam is said to be {\em weakly embeddable} (resp. {\em weakly poembeddable}) in a pomonoid $W$ if there exist monomorphisms (resp. order embeddings) $\theta_{i}: S_{i}\to W$ such that  $\theta_{i}\varphi_{i} =  \theta_{j}\varphi_{j}$  for all $i \neq j$ in $I$. If in addition $\theta_{i}(S_{i}) \cap \theta_{j}(S_{j}) =  \theta_{i}\varphi_{i}(U)$ then we say that the amalgam is {\em strongly embeddable} (resp. {\em strongly poembeddable}) in $W$. It is worth noting that a pomonoid monomorphism is just  a one-to-one monotone monoid homomorphism.

\medskip

We shall see later that the embeddability of amalgams of pomonoids is closely connected with {\em tensor products} of certain $S-$posets. Let $S$ be a pomonoid, let $A$ be a right $S-$poset and let $B$ be a left $S-$poset. The cartesian product $A\times B$ can be endowed with a partial order defined componentwise. Then the {\em tensor product of $A$ and $B$ over $S$} is the poset given by $A\otimes_S B = (A\times B)/\tau$ where $\tau$ is the order-congruence on $A\times B$ (considered as an $S-$poset with trivial $S-$action) generated by
$$
H=\{((as,b),(a,sb))| a\in A, b\in B, s\in S\}.
$$
When appropriate we normally denote the tensor product by $A\otimes B$ and the congruence class $(a,b)\tau = a\otimes b$.

The order on the  tensor product $A \otimes B$ is given as follows (see~\cite[Theorem 5.2]{shi-2005} for more details).
For $a,a'\in A, b,b'\in B$, $a \otimes b \le a' \otimes b'$ if and only if there exists $n\ge 1$ and $a_1, \dots, a_n \in A, b_2, \dots, b_n \in B$, and $s_1, \dots, s_n, t_1, \dots, t_n \in S$ such that
\begin{align*}
a& \le a_{1}s_{1} &   s_{1} b&  \le t_{1} b_{2}\\
a_{1}t_{1} & \le a_{2}s_{2} &  s_{2}b_{2} &  \le t_{2} b_{3}\\
\vdots &  & \vdots\\
a_{n-1}t_{n-1} & \le a_ns_n & s_{n}b_{n}&  \le t_{n} b'\\
a_{n}t_{n} & \le a'.
\end{align*}

If $\lambda$: $A \to B$ is a left $S-$poset morphism and if $Y$ is a right $S-$poset such that $y \otimes a\le y'\otimes a'$ in $Y\otimes A$ then it is easy to see that $y\otimes\lambda(a)\le y'\otimes\lambda(a')$ in $Y\otimes B$. In addition $y \otimes s \le y' \otimes s'$ in $Y \otimes_S S$ if and only if  $ys \le y's'$ (see~\cite[Corollary 3.3]{shi-2008}). 

We can also define tensor products of $S-$posets in terms of {\em balanced maps} in the usual way (see~\cite[Theorem 5.3]{shi-2005} and~\cite[Proposition 8.1.10]{howie-1995}) and so we can easily deduce,

\begin{lemma}\label{associative-lemma}{{\rm [Cf.~\cite[Proposition 8.1.11]{howie-1995}]}}
Let $S$ and $T$ be pomonoids and let $A$ be a right $S-$poset, $B$ an $(S,T)-$poset and $C$ and left $T-$poset. Then $(A\otimes_SB)\otimes_TC$ is order isomorphic to $A\otimes_S(B\otimes_TC)$.
\end{lemma}

In addition

\begin{lemma}\label{order-lemma}
Suppose that $R,S,T$ are pomonoids, $X$ is an $(R,S)-$poset and $Y$ an $(S,T)-$poset. Then
\begin{enumerate}
\item if $x\le x'$ in $X$ and $y\le y'$ in $Y$ then $x\otimes y\le x'\otimes y'$ in $X\otimes_SY$,
\item if $x\otimes y\le x'\otimes y'$ in $X\otimes_SY$ and $t\le t'\in T,r\le r'\in R$ then $rx\otimes yt\le r'x'\otimes y't'$ in $X\otimes_SY$.
\end{enumerate}
\end{lemma}
\begin{proof} The first part follows easily from the following scheme
\begin{align*}
x& \le x'1 &   1y&  \le 1y'\\
x'1& \le x'.
\end{align*}
For the second part it follows that $r\otimes(x\otimes y)\otimes t\le r'\otimes(x'\otimes y')\otimes t$ in $R\otimes_R(X\otimes_SY)\otimes_TT$ and so from the observation before Lemma~\ref{associative-lemma}, and its dual, we deduce that $rx\otimes yt\le r'x'\otimes y't'$ in $X\otimes_SY$.
\end{proof}

\bigskip

Let $I$ be a quasi-ordered set. A {\em direct system} of right $S-$posets $(X_{i}, \varphi^{i}_{j})_{i\in I}$ is a collection of right $S-$posets $X_{i}$ and a collection of $S-$poset morphism $\varphi^{i}_{j}$: $X_{i} \to X_{j}$, $i \le j$, which satisfies
\begin{enumerate}
\item $\varphi^{i}_{i}$ = $1_{X_{i}}$, and 
\item $\varphi^{j}_{k} \circ \varphi^{i}_{j}$ = $\varphi^{i}_{k}$ whenever $i \le j \le k$.
\end{enumerate}
The {\em direct limit} of this direct system is an $S-$poset $X$ and $S-$poset morphisms $\varphi_{i}$: $X_{i} \to X$ which  satisfy
\begin{enumerate}
\item $\varphi_{j} \circ \varphi^{i}_{j}$ = $\varphi_{i}$ whenever $i \le j$,
\item if $Y$ is an $S-$poset and $\psi_{i}$: $X_{i} \to Y$ are $S-$poset morphisms such that $\psi_{j} \circ \varphi^{i}_{j}$ = $\psi_{i}$ whenever $i \le j$, then there exists a unique $S-$poset morphism $\psi$: $X \to Y$ such that  $\psi \circ \varphi_{i}$ = $\psi_{i}$ for all $i\in I$.
\end{enumerate}
It is straightforward to demonstrate that the direct limit of any direct system of $S-$posets exists and is unique up to isomorphism (see for example~\cite{bf-2005}).

\smallskip

The set $I$ is called {\em directed} if for all $i, j \in I$, there exists $k \in I$ such that $k \ge i, j$. The first part of the following lemma appears in \cite{bf-2005} as Proposition 2.5, the other parts are straightforward.

\begin{lemma} \label{direct-limit-lemma}
Let $(X_{i}, \varphi_{j}^{i})$ be a direct system in the category of $(S, T)-$posets with directed index set and let $(X, \alpha_{i})$ be the direct limit of this system. Then
\begin{enumerate}
\item $\varphi_{i}(x_{i}) \le \varphi_{j}(x_{j})$ in $X$ if and only if there exists $k  \ge i, j$ such that $\varphi_{k}^{i}(x_{i}) \le \varphi_{k}^{j}(x_{j})$;
\item the map $\varphi_{i}$ is one-to-one if and only if $\varphi_{k}^{i}$ is one to one for all $k\ge i$;
\item the map $\varphi_{i}$ is order embedding if and only if  $\varphi_{k}^{i}$ is an order embedding for all $k\ge i$.
\end{enumerate}
\end{lemma} 

As a special case, the direct limit in the category of $S-$posets of the diagram 

$$
\begin{diagram}
\node{A}  \arrow{e,t}{f} \arrow{s,t}{g} \node{B}  \\
\node {C}  
\end{diagram}
$$
is called the {\em pushout} of the diagram. It is easy to show that it is isomorphic to the quotient of the coproduct $F = B \dot\cup C$ by the $S-$poset congruence $\rho$ generated by
$$
R = \lbrace (f(a), g(a)) : a \in A \rbrace.
$$
Recall (see \cite{bf-2006}) that the coproduct of $B$ and $C$ is their disjoint union with component-wise order.
The associated maps $\gamma : B \to F/\rho$ and $\delta : C \to F/\rho$ are given by $\gamma(b)= b\rho$ and $\delta(c) = c\rho$ respectively. As in the category of $S-$acts, it is easy to demonstrate that tensor products preserve pushouts.

\begin{lemma}  \label{bad1}
Let $$
\begin{diagram}
\node{A}  \arrow{e,t}{f} \arrow{s,t}{g} \node{B} \arrow{s,b}{\gamma}  \\
\node {C}  \arrow{e,b}{\delta}  \node {D}
\end{diagram}
$$
be a pushout in the category of $S-$posets. If $\gamma(b) \le \delta(c)$, where $b \in B$ and $c \in C$, then there exists $a, a' \in A$ such that $b \le  f(a)$ and $g(a') \le c$.  
\end{lemma}

\begin{proof}
Suppose $\gamma(b) \le \delta(c)$. Then, since $D =  (B \dot\cup C) / \rho$ where $\rho$ is the $S-$poset congruence generated by
$$
R = \lbrace (f(a), g(a)) : a \in A \rbrace,
$$
it follows that $b\rho \le c\rho$ in $D$ and so $b \le_{\alpha(R \cup R^{-1})} c$. Now $b \le_{\alpha(R \cup R^{-1})} c$ if and only if there exists $d_1,d'_1,\ldots,d_n,d'_n \in B\dot\cup C$ such that
\begin{equation*}
b \le d_1  \;\alpha(R \cup R^{-1}) \; d'_1 \le d_2 \; \alpha(R \cup R^{-1}) \; d'_2 \le d_3 \dots d_n \;\alpha(R \cup R^{-1}) \; d'_n \le c.\tag{$\ast$}
\end{equation*}
Since $b\le c$ is impossible, we can assume that we have a system such as ($\ast$) of minimal length $n\ge 1$ such that $d_1\ne d'_n$.
Since  $b \le d_1$ and $d'_n \le c$, $d_1 \in B$ and $d'_n \in C$ by definition of the order on $B\dot\cup C$. For $ d_i  \;\alpha(R \cup R^{-1}) \; d'_i$, $1\le i\le n$, minimality of $n$ allows us to deduce that there exists $d_{i_1}, d'_{i_1},\ldots,d_{i_m},d'_{i_m}$ such that
$$
d_{i} = d_{i_1}s_{i_1},\;  d'_{i_1}s_{i_1} = d_{i_2}s_{i_2},\; \dots ,\; d'_{i_{m}}s_{i_{m}} = d'_{i},\text{ where }(d_{i_j}, d'_{i_j}) \in R \cup R^{-1}.
$$
Consequently we deduce that $d_1=f(a)$ for some $a\in A$. Similarly $d'_n=g(a')$ for some $a'\in A$ and so the result follows.
\end{proof}

\medskip

A subposet $X$ of a poset $P$ is called {\em convex}, if for any $x, y\in X, z\in P$ with $x\le z\le y$, $z\in X$. If $f : X \to Y$ is an $S-$poset morphism then we shall say that {\em $f$ is convex} if $\im(f)$ is convex in $Y$.

\begin{lemma} \label{bad}
Let $$
\begin{diagram}
\node{A}  \arrow{e,t}{f} \arrow{s,t}{g} \node{B} \arrow{s,b}{\gamma}  \\
\node {C}  \arrow{e,b}{\delta}  \node {D}
\end{diagram}
$$
be a pushout in the category of $S-$posets.
\begin{enumerate}
\item If  $\gamma(b) = \delta(c)$, $b\in B, c\in C$ then there exists $a_1, a'_1, a_2, a'_2 \in A$ such that $f(a'_1)\le b \le  f(a_1)$ and $ g(a'_2) \le c\le g(a_2)$;
\item if $f$ and $g$ are convex and if $\gamma(b) = \delta(c)$, $b\in B, c\in C$ then there exists $a, a' \in A$ such that $b = f(a) $ and $ c = g(a')$;
\item if $f$ is an order embedding then $\delta$ is also an order embedding;
\item if  $f$ and $g$ are convex and order embeddings and if $\gamma(b) = \delta(c)$ with $b \in B, c \in C$, then there exists a unique $a \in A$ such that $b = f(a) $ and $ c = g(a)$; 
\item if $f$ is convex then $\delta$ is convex.
\end{enumerate}
\end{lemma}

\begin{proof} Part (1) follows easily from Lemma~\ref{bad1} and part (2) is straightforward. For part (3), suppose that $\delta(c)\le\delta(c')$ so that $c\rho\le c'\rho$. It follows that $c\le_{\alpha(R\cup R^{-1})}c'$ and so there exists $n\ge1, d_i,d'_i\in D$ and $1\le i\le n$ such that
$$
c\le d_1\alpha(R\cup R^{-1})d'_1\le d_2 \ldots d'_{n-1}\le d_n\alpha(R\cup R^{-1})d'_n\le c'.
$$
In addition, we can assume that this sequence is of minimal length. It follows that if for some $1\le i\le n$, $d_i=d'_i$ then $c\le c'$, so let us assume that for each $i$, $d_i\ne d'_i$. It is clear that $d_1\in C$ and that there exists $m\ge 1$ and $x_j,x'_j,s_j, 1\le j\le m$ such that
$$
d_1=x_1s_1, x'_1s_1=x_2s_2, \ldots, x'_ms_m=d'_1,
$$
where for each $j$, $(x_j,x'_j)\in R\cup R^{-1}$. Moreover we can assume that this sequence of equations is also of minimal length. Suppose that $m\ge 2$. Since $d_1\in C$, it follows that $x_1=g(a_1), x'_1=f(a_1), x_2=f(a_2), x'_2=g(a_2)$ for some $a_1,a_2\in A$. Since $f$ is an order embedding, $a_1s_1=a_2s_2$ and so $d_1=g(a_1)s_1 =x'_2s_2$. This contradicts the minimality of $m$ and so $m=1$. This means that $d_1=g(a_1)s_1, f(a_1)s_1=d'_1$. Using a similar argument we can deduce that $d_2=f(a'_1)s'_1, g(a'_1)s'_1=d'_2$ for some $a'_1\in A, s'_1\in S$. Since $d'_1\le d_2$ and $f$ is an order embedding, we deduce that $a_1s_1\le a'_1s'_1$ and so $d_1\le d'_2$. This contradicts the minimality of $n$ and so $c\le c'$ as required.

To see (4), it follows from (2) that there exists $a,a'\in A$ with $b=f(a), c=g(a')$. Hence
$$
\gamma f(a) = \gamma(b) = \delta(c)=\delta g(a') = \gamma f(a'),
$$
and so $a=a'$ by (3). Uniqueness of this $a\in A$ also follows from (3).

For (5), suppose that $f$ is convex and $\delta(c) \le d \le \delta(c')$ where $d \in D$. Then either $d=b_1 \rho$ or $d=c_1 \rho$ with $b_1 \in B$ and $c_1 \in C$.  In the latter case $d$ is in the image of $\delta$ and so  $\delta$ is a convex map.  Otherwise $d=b_1 \rho = \gamma(b_1)$ and so $\delta(c) \le \gamma(b_1)\le\delta(c')$. From Lemma \ref{bad1} there exist $a_1, a'_1, a_2, a'_2 \in A$ such that
\begin{center}
$c \le g(a_1) \text{  } \text{   }f(a'_1) \le b_1$\\
$b_1 \le f(a_2) \text{  } \text{   }g(a'_1) \le c'.$
\end{center}
Since $f$ is convex, $b_1 = f(a)$ for some $a \in A$. Hence $d=\gamma(b_1) =\gamma f(a) = \delta g(a)$ and so $\delta$ is convex.
\end{proof}

\section{Free extensions, free products and amalgamation}

Let $U$ be a subpomonoid of a pomonoid $S$ and let $Y$ be a right $U-$poset and $X$ be a right $S-$poset such that $f: X \to Y$ is right $U-$poset morphism. The {\em free $S-$extension of $X$ and $Y$} is a right $S-$poset $F=F(S;X,Y)$ together with a $U-$poset morphism $g: Y\to F$ such that
\begin{enumerate}
\item $h=gf:X\to F$ is an $S-$poset morphism;
\item whenever there is a right $S-$poset $Z$ and a right $U-$poset morphism $\alpha:Y\to Z$ where $\beta = \alpha f:X\to Z$ is a right $S-$poset morphism then there exists a unique right $S-$poset morphism $\psi: F\to Z$ such that
$\psi g = \alpha$ and $\psi h = \beta$.
\end{enumerate}

\begin{theorem} \label{b1}
Free $S-$extensions exist in the category of $S-$posets and are unique up to isomorphism.
\end{theorem} 

\begin{proof} Suppose that  $X$ is a right $S-$poset, $Y$ is a right $U-$poset, and  $f: X\to Y$ is a right $U-$poset morphism. Notice that $Y\otimes S$ is a right $S-$poset with action given by $(y\otimes s)t = y\otimes st$. Suppose that $\rho = \nu(R)$ is the right $S-$poset congruence on $Y\otimes_{U}S$ induced by the relation:
$$
R=\lbrace (f(x)\otimes s, f(x')\otimes s'): xs\le x's'\rbrace.
$$
Let $g: Y \to (Y\otimes S)/\rho$ be given by $g(y) = (y\otimes 1)\rho$. Then, it is straightforward, as in the unordered case (see~\cite[Theorem 4.18]{renshaw-1986}), to show that $((Y\otimes S)/\rho,g)$ is the free $S-$extension of $X$ and $Y$ and that it is unique up to isomorphism.
\end{proof}
Notice that if $(y\tensor s)\rho \in (Y\tensor S)/\rho$ then $(y\tensor s)\rho=g(y)s$. From this we can easily deduce the following useful result.

\begin{lemma}\label{free-extension-order-lemma}
Let $U$ be a subpomonoid of a pomonoid $S$, $X$ be a right $S-$poset, $Y$ be a right $U-$poset, and  $f:X\to Y$ be a right $U-$poset morphism. Let $F=(Y\tensor_U S)/\rho$ be the free $S-$extension of $X$ and $Y$. If $y\le y'$ in $Y$ and $s\le s'$ in $S$ then $(y\tensor s)\rho\le (y'\tensor s')\rho$ in $F$.
\end{lemma}

From now on, unless specifically mentioned, all tensor products will be over $U$. As in the category of $S-$acts,  it is possible to define the free $S-$extension in terms of pushouts.

\begin{lemma} \label{b2}
Let $U$ be a subpomonoid of a pomonoid $S$, $X$ be a right $S-$poset, $Y$ be a right $U-$poset, and  $f:X\to Y$ be a right $U-$poset morphism. Then, the free $S-$extension $F$ of $X$ and $Y$ is the pushout in the category of right $S-$posets of the diagram
$$
\begin{diagram}
\node{X\otimes S}  \arrow{e,t}{f\otimes 1} \arrow{s,l}{\varphi} \node{Y\otimes S} \\
\node{X  } 
\end{diagram}
$$
where $\varphi :X\otimes S\to X$ is given by $\varphi(x\otimes s) = xs$.
\end{lemma}

The proof is straightforward and details are left to the reader.

\medskip

Given a family $\{S_i:i\in I\}$ of pairwise disjoint posemigroups, the {\em free product}  ${\mathcal F}=\prod^\ast{S_i}$ of the family $\{S_i:i\in I\}$ is the set of non-empty words
$$
a_1\ldots a_n
$$
with each $a_k\in S_i$ for some $i\in I, 1\le k\le n$, and no two adjacent letters are in the same $S_i$. The product in ${\mathcal F}$ is defined by 
$$
(a_1\ldots a_n)(b_1\ldots b_m)=\begin{cases}a_1\ldots a_nb_1\ldots b_m&\text{ if }a_n\in S_i, b_1\in S_j, i\ne j\\a_1\ldots (a_nb_1)\ldots b_m&\text{ if }a_n,b_1\in S_i.\\\end{cases}
$$
Similarly we can define an order on ${\mathcal F}$ by $a_1\ldots a_r\le b_1\ldots b_s$ if and only if
\begin{enumerate}
\item $r=s$,
\item for each $1\le i\le r$, $a_i,b_i\in S_j$ for some $1\le j\le n$ and $a_i\le b_i$ in $S_j$.
\end{enumerate}
It is easy to check that ${\mathcal F}$ is then a posemigroup and that ${\mathcal F}$ together with the order embeddings $\theta_i:S_i\to {\mathcal F}$ given by $\theta_i(s_i)=s_i$, is the coproduct in the category of posemigroups of the family $\{S_i:i\in I\}$.

\medskip

Now suppose that
$\varphi_i:U\to S_i$ are posemigroup morphisms. Let $\sigma = \nu(R\cup R^{-1})$ be the posemigroup congruence on ${\mathcal F}$ generated by
$$
R=\{(\varphi_i(u),\varphi_j(u)):u\in U, i,j\in I\}.
$$
Define maps $\mu_i:S_i\to {\mathcal F}/\sigma$ by $\mu(s_i)=s_i\sigma$. Then we can easily check that $P={\mathcal F}/\sigma$ together with the $\mu_i$ is the pushout in the category of posemigroups of the family $\{S_i,\varphi_i:i\in I\}$.

\smallskip

Suppose that $w,w'$ are in ${\mathcal F}$ and consider the following four types of transitions. We say that $w$ is connected to $w'$ by an
\begin{enumerate}
\item {\em $S-$step} if $w=(s_1,\ldots,s_{i-1},u,s_{i+1},\ldots,s_n), \ w'=(s_1,\ldots,s_{i-1}us_{i+1},\ldots,s_n)$;
\item {\em $M-$step} if $w=(s_1,\ldots,s_{i-1}us_{i+1},\ldots,s_n), \ w'=(s_1,\ldots,s_{i-1},u,s_{i+1},\ldots,s_n)$;
\item {\em $E-$step} if any one of the following holds
\begin{enumerate}
\item $w=(s_1,\ldots,s_iu,s_{i+1},\ldots,s_n), \ w'=(s_1,\ldots,s_i,us_{i+1},\ldots,s_n)$;
\item $w=(s_1,\ldots,s_i,us_{i+1},\ldots,s_n), \ w'=(s_1,\ldots,s_iu,s_{i+1},\ldots,s_n)$;
\item $w=(s_1,\ldots,s_nu), \ w'=(s_1,\ldots,s_n,u)$;
\item $w=(s_1,\ldots,s_n,u), \ w'=(s_1,\ldots,s_nu)$;
\item $w=(us_1,\ldots,s_n), \ w'=(u,s_1,\ldots,s_n)$;
\item $w=(u,s_1,\ldots,s_n), \ w'=(us_1,\ldots,s_n)$;
\end{enumerate}
\item {\em $O-$step} if $w=(s_1,\ldots, s_i,\ldots, s_n), \ w'=(s_1,\ldots,s'_i,\ldots, s_n)$, where $s_i\le s'_i$.
\end{enumerate}
In a manner similar to~\cite{howie-1962}, it is straightforward to show that $w\sigma\le w'\sigma$ in $P={\mathcal F}/\sigma$ if and only if $w$ is connected to $w'$ by a finite sequence of $E-$, $S-$, $M-$ or $O-$steps. We shall make use of this later.

\medskip

In general we shall restrict our attention to the case when $|I|=2$. When $\varphi_{1}$ and $\varphi_{2}$ are order embeddings, we normally refer to the pushout ${\mathcal F}/\sigma$ as the {\em amalgamated free product} and denote it by  $S_{1}\ast_{U}S_{2}$. Notice that this is the same notation as in the unordered context but no confusion should arise.

If $U, S_1$ and $S_2$ are pomonoids and $\varphi_i$ pomonoid morphisms, then by identifying the identity elements of $S_1$ and $S_2$ within $\mathcal F$, we obtain the coproduct in the category of pomonoids and a construction similar to the above one gives the amalgamated free product in the category of pomonoids.

If $S$ is a posemigroup and if we denote by ${}^1S$ the monoid obtained by adjoining an identity 1 to $S$ regardless of whether $S$ already has an identity, then ${}^1S$ becomes a pomonoid if we extend the ordering on $S$ to ${}^1S$ by considering $1$ as an incomporable element in $^1S$.
As in the unordered case, (see~\cite{renshaw-1986}), it is straightforward to show that if $S_1\ast_U S_2$ is the amalgamated free product in the category of posemigroups of the posemigroup amalgam $[U;S_1,S_2]$ then ${}^1(S_1\ast_U S_2)$ is isomorphic to ${}^1S_1\ast_{{}^1U}{}^1S_2$, the amalgamated free product in the category of pomonoids of the pomonoid amalgam $[{}^1U;{}^1S_1,{}^1S_2]$. Consequently, from now on we shall only deal with pomonoid amalgams.

In~\cite{sohail-2011}, the above amalgamated free product is obtained by first endowing the corresponding monoid amalgamated free product with trivial order and then factoring it by an order congruence.

\medskip

{\bf Remark.}
As in every category in which pushouts exist, the pomonoid amalgam $A$ = $[U$; $S_{1}$, $S_{2}]$ is embeddable if and only if it is naturally embeddable in its free product.

\bigskip

Let $[U; S_1, S_2]$ be an amalgam of pomonoids. We define a direct system of $U-$posets $(Y_n,k_n)$ whose direct limit is isomorphic to $S_1\ast_U S_2$. The process is very similar to that in the unordered case and we direct the reader to~\cite{renshaw-1986} for more details. Let $Y_{1} = S_{1}$, $Y_{2} = S_{1} \otimes S_{2}$, and $k_1:Y_1\to Y_2$ be given by $k_1(s_1)=s_1\tensor1$. By way of induction, assume that we have constructed a sequence $Y_1,Y_2, \ldots, Y_{n-1}$ with maps $f_i:Y_i\to Y_{i+1}$ for $i=1,\ldots, n-2$. Define $Y_{n} = F(S_i; Y_{n-2}, Y_{n-1}) = (Y_{n-1} \otimes S_{i})/\delta_{n-2}, i\equiv n$ (mod 2),  where $\delta_{n-2}$ is the $S_i-$poset congruence induced on  $Y_{n-1} \otimes S_{i}$  by 
$$
V_{n-2}=\lbrace(k_{n-2}(y_{n-2})\otimes s_{i}, k_{n-2}(y'_{n-2})\otimes s'_{i}): y_{n-2}s_{i} \le y'_{n-2}s'_{i}\rbrace
$$
and let $k_{n-1}:Y_{n-1}\to Y_n$ be the associated $U-$poset morphism defined by $k_{n-1}(y_{n-1}) = (y_{n-1} \otimes 1)\delta_{n-2}$.
Then $(Y_{n}, k_{n})_{n\ge 1}$ is a direct system in the category of $U-$posets.

As with the unordered case, a typical element of $Y_n$ is
$$
y_n=(\ldots((s_1\tensor s_2\tensor s_3)\delta_1\tensor s_4)\delta_2\tensor\ldots\tensor s_n)\delta_{n-2}.
$$
We shall denote this by $y_n=[s_1, \dots, s_n]$ and a typical element of  $S_1 \ast_U S_2$ by $(s_1, \dots, s_n)$.

\begin{lemma}
For all $i \ge 2$,
$$[s_1, \dots, s_{i-1}, 1, s_{i+1}] = [s_1, \dots, s_{i-1} s_{i+1}, 1, 1]$$
\end{lemma}
\begin{proof}
Suppose $y_{i-1} = [s_1, \dots, s_{i-1}] \in Y_{i-1}$. Then, 

 \begin{equation*}
\begin{split}
[s_1, \dots, s_{i-1}, 1, s_{i+1}] & = ((y_{i-1} \otimes 1) \delta_{i-2} \otimes s_{i+1})\delta_{i-1}\\
& = (k_{i-1}(y_{i-1}) \otimes s_{i+1})\delta_{i-1}\\
& = (k_{i-1}(y_{i-1} s_{i+1}) \otimes 1)\delta_{i-1}\\
& = ((y_{i-1} s_{i+1} \otimes 1) \delta_{i-2} \otimes 1)\delta_{i-1}\\
& = [s_1, \dots, s_{i-1} s_{i+1}, 1, 1]
\end{split}
 \end{equation*}
\end{proof}
Consequently we can deduce
\begin{corollary} \label{free-extension-word-corollary}
For all $i \ge 2$, $[s_1, \dots, s_{i-1}, 1, s_{i+1}, \dots, s_n] = [s_1, \dots, s_{i-1} s_{i+1}, \dots, s_n, 1, 1]$.
\end{corollary}

In addition, from Lemma~\ref{free-extension-order-lemma} we can easily deduce

\begin{lemma} \label{yn-order-lemma}
For all $i \ge 2$, if $s_i\le t_i$ in $S_j$, $(j\in\{1,2\}, j\equiv i\ \mod\ 2)$ then
$$
[s_1, \dots, s_{i-1}, s_i, s_{i+1}, \dots, s_n] \le [s_1, \dots, s_{i-1},t_i, s_{i+1}, \dots, s_n].
$$
\end{lemma}

\medskip

The proof of the following major result is almost identical to that in the unordered case and is therefore omitted. For the interested reader, more details can be found in~\cite[Theorem 1]{sohail-2014} and~\cite[Theorem 3.3.9]{al-subaiei-2014}.

\begin{theorem} \label{direct-limit-theorem}
Let $[U, S_{1}, S_{2}]$ be an amalgam of pomonoids. Then, $(S_{1}\ast_US_{2},\varphi_n)$ is the direct limit in the category of $(U, U)-$posets of the direct system $(Y_{n}, k_{n})_{n\ge 1}$ where $\varphi_n:Y_n\to S_1\ast_US_2$ is given by $\varphi_n([s_1,\ldots,s_n])=(s_1,\ldots,s_n)$.
\end{theorem}

\bigskip

If we define, for $n\ge 2$, $k^{(n-1)}$ = $k_{n-1} \circ k_{n-2} \circ \dots \circ k_{1} : Y_{1} \to Y_{n}$,  $h^{1}: S_{2} \to Y_{2}$, and  $h^{(n-1)}$ = $k_{n-1} \circ k_{n-2} \circ \dots \circ k_{2} \circ h^{1}: S_{2} \to Y_{n}$, then it is straightforward to show the following results.

\begin{lemma} \label{poamalgam-weak-embeddability-condition-lemma}
The pomonoid amalgam $A = [U; S_{1}, S_{2}]$ is weakly embeddable (resp. poembeddable) if and only if for all $n \ge 1$ the maps  $k^{n}$ and $h^{n}$ are monomorphisms (resp. order embeddings).
\end{lemma}

\begin{lemma} \label{poamalgam-strong-embeddability-condition-lemma}
Let the pomonoid amalgam $[U; S_1, S_2]$ be weakly embeddable (resp. poembeddable) and $\varphi_2: Y_2 \to S_{1}\ast_US_{2}$ be one to one. Then, the pomonoid amalgam is strongly embeddable (resp. strongly poembeddable) if and only if $s_1 \otimes 1 = 1 \otimes s_2 $ in $Y_2$ implies $s_1 = s_2 \in U$.
\end{lemma}

\bigskip

Let $U$ be a subpomonoid of a pomonoid $S$. Then $U$ has the {\em right poextension (resp. extension) property} in $S$ if the map $X\to X\otimes S, x\mapsto x\otimes1$ is an order embedding (resp. monomorphism) for every right $U-$poset $X$. The {\em left poextension property} is defined dually. We say that $U$ has the {\em poextension (resp. extension) property} in $S$ if for every left $U-$poset $X$ and right $U-$poset $Y$ the map $X\otimes Y\to X\otimes S\otimes Y, x\otimes y\mapsto x\otimes1\otimes y$ is an order embedding (resp. monomorphism). If $U$ has the poextension (resp. extension) property in every pomonoid containing $U$, then we shall say that $U$ is {\em absolutely poextendable (resp. extendable)}.

\smallskip

Let $U$ be a subpomonoid of a pomonoid $S$. We shall say that the pair $(U,S)$ is a {\em weak amalgamation (resp. poamalgamation) pair} if for every pomonoid $T$ containing $U$ the amalgam $[U;S,T]$ is weakly embeddable (resp. poembeddable). $U$ is called a {\em weak amalgamation (resp.  poamalgamation) base} if for every pomonoid $S$, $(U,S)$ is a  weak amalgamation (resp. poamalgamation) pair.

\begin{theorem} \label{we15}
Let $(U; S)$ be a weak poamalgamation (resp. amalgamation) pair in the category of pomonoids. Then $U$ has the poextension (resp. extension) property in $S$.
\end{theorem}

\begin{proof}
The other argument being similar, we prove this for poamalgamation pairs. 

Let $X$ be a right $U-$poset and $Y$ a left $U-$poset. Suppose $Z = X \dot\cup Y$ with componentwise order and extend the action of $U$ on $X$ and $Y$ to a bi-action on $Z$ as follows. Let $u.x = x$, $y.u = y$ for all $x \in X$, $y \in Y$, $u \in U$, and $x.u$ (resp. $u.y$) is evaluated in $X_U$ (${}_UY$). It is easy to check that $Z$ is a $(U, U)-$poset. Let $Z^{(0)} = U$, $Z^{(1)} = Z$, and $Z^{(n)} = Z^{(n-1)} \otimes_U Z$ for all $n \ge 2$. Let $T = \dot\cup_{n \ge 0}  Z^{(n)}$ and extend the multiplication of $U$ to $T$ by
$$
\begin{array}{rcl}
(z_1 \otimes \dots \otimes z_m).(w_1 \otimes \dots \otimes w_n)&=& z_1 \otimes \dots \otimes z_m \otimes w_1 \otimes \dots \otimes w_n,\\
u.(z_1 \otimes \dots \otimes z_m)& =& (uz_1) \otimes \dots \otimes z_m,\\
(z_1 \otimes \dots \otimes z_m).u &=& z_1 \otimes \dots \otimes (z_mu).
\end{array}
$$
It is clear that $T$ is a monoid. Also $T$ is a poset with order induced from the order on the components $Z^{(n)}$.
The compatibility of this order follows easily from Lemma~\ref{order-lemma}.

Hence $U$ is a subpomonoid of the pomonoid $T$ and so $[U; S, T]$ is a pomonoid amalgam which, by assumption, is weakly embeddable in its amalgamated free product. In other words the the following diagram commutes:
$$
\begin{diagram}
\node{ U}  \arrow{e,t}{\varphi_1} \arrow{s,l}{\varphi_2}  \node{S}  \arrow{s,r}{\lambda_1}\\
 \node{T}  \arrow{e,b}{\lambda_2} \node{S \ast_U T}
\end{diagram}.
$$ 

Suppose $x \otimes 1 \otimes y$ $\le$ $x' \otimes 1 \otimes y'$ in $X \otimes S \otimes Y$. Then, since the map $X \otimes S \otimes Y \rightarrow T \otimes S \otimes  T$ given by $x \otimes s \otimes y \mapsto x \otimes s \otimes y$ is monotone, it follows that $x \otimes 1 \otimes y \le x' \otimes 1 \otimes y'$ in $T \otimes S \otimes T$. Consequently, because the map $T \otimes  S \otimes  T \rightarrow S \ast_U T$ given by $t \otimes s \otimes t'\mapsto \lambda_2(t) \lambda_1(s) \lambda_2(t')$ is also monotone, we deduce that $\lambda_2(x) \lambda_1(1) \lambda_2(y)  \le  \lambda_2(x')  \lambda_1(1) \lambda_2(y')$ in $S \ast_U T$ and so $\lambda_2(xy) \le \lambda_2(x'y')$. But $\lambda_2$ is an order embedding and so $xy \le x'y'$ in $T$.
Now the  map $X \otimes Y \rightarrow T$  given by $x \otimes y \mapsto xy$ is an order embedding since if $xy \le x'y'$ in $T$ then from the definition of the multiplication on $T$ we have
$$
x \otimes y \le x' \otimes  y'\text{ in }Z\otimes_UZ\cong (X\dot\cup Y)\otimes_U(X\dot\cup Y)\cong (X\otimes_U X)\dot\cup(X\otimes_UY)\dot\cup(Y\otimes_UX)\dot\cup(Y\otimes_UY)
$$
and so $x \otimes y \le x' \otimes  y'$ in $X\otimes Y$ as required.
\end{proof}

\begin{corollary}\label{poamalgamation-base-lemma}
Let $U$ be a weak poamalgamation (resp. amalgamation) base in the category of pomonoids. Then $U$ is absolutely  poextendable (resp. extendable).
\end{corollary}

\section{Pounitary subpomonoids and amalgmation}
The concept of a {\em unitary subsemigroup} has been known to be related to the question of embeddability of amalgams since Howie's pioneering work in~\cite{howie-1962}. Gould and Shaheen~\cite{gould-2010} generalised this concept for posemigroups during their study of projective covers of pomonoids. We generalise this even further and provide a number of connections with amalgamation.

Let $U$ be a subpomonoid of the pomonoid $S$ and let $v, u, u_1, u'_1, \dots  u_n, u'_n \in U$, $s, s_1, s_2, \dots s_n \in S$. We shall say that
\begin{enumerate}
\item {\em $U$ is upper strongly right pounitary in $S$} (USRPU) if $v \le s u \Rightarrow s \in U$;
\item {\em $U$ is lower strongly right pounitary in $S$} (LSRPU) if $s u \le v \Rightarrow s \in U$;
\item {\em $U$ is strongly right pounitary in $S$} (SRPU) if ($v \le s u \text{ or } s u \le v) \Rightarrow s \in U$;
\item {\em $U$ is right pounitary in $S$} (RPU) if whenever there exists $n\ge 1$ such that
$$
u \le s_1 u_1, s_1 u'_1 \le s_2 u_2, \dots s_n u'_n \le v
$$
then $s_1 ,  s_2, \dots,  s_n \in U$;
\item {\em $U$ is right unitary in $S$} (RU) if $s u = v \Rightarrow s \in U$.
\end{enumerate}

Left-handed versions of these conditions are defined in a dual manner. If both the right and left handed version hold then we shall omit the adjective altogether. The implications represented by the following diagram are fairly clear.

$$
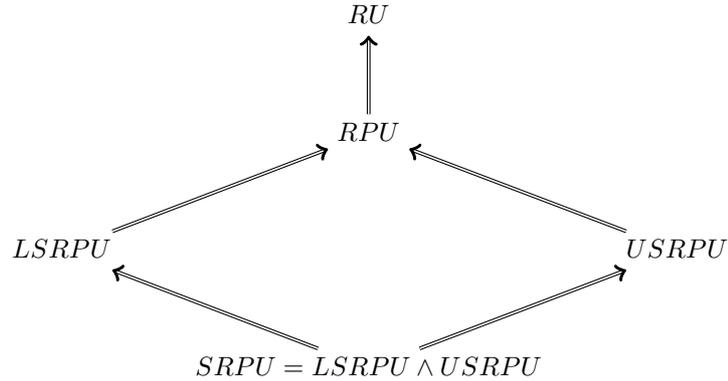
\begin{tikzpicture}[description/.style={fill=white,inner sep=2pt}]
\matrix (m) [matrix of math nodes, row sep=3em,
column sep=2.5em, text height=1.5ex, text depth=0.25ex]
{&RU&\\
&RPU&\\
LSRPU&  & USRPU \\
& SRPU=LSRPU \wedge USRPU&\\ };
\path[->]
(m-2-2) edge [double]  (m-1-2)
(m-3-1) edge [double] (m-2-2)
(m-3-3) edge [double] (m-2-2)
(m-4-2) edge [double] (m-3-1)
(m-4-2) edge [double] (m-3-3);
\end{tikzpicture}
$$

The implications are strict as the following examples demonstrate. In~\cite{gould-2010} it is shown that if a pomonoid $U$ is right unitary in $S$ then it need not be right pounitary in $S$.

Let $U = \lbrace 0, 1, 2, 3, \dots n \rbrace \subseteq (\mathbb{N}^{0}, \max)$. Then clearly $U$ is lower strongly pounitary in $\mathbb{N}^0$ but not upper strongly pounitary in $\mathbb{N}^0$. 

Now suppose that $U = \lbrace 1, e, f \rbrace$ is a subpomonoid of a pomonoid $S$  

\begin{center}
\begin{tabular}{ l | c c c  c c r }
$S$ &  $a$ & $f$ & $b$  & $e$ &  $1$   \\ \hline
$a$ &  $a$ & $a$ & $a$  & $a$ &  $a$ \\
$f$  &  $a$ & $f$ & $b$  & $f$ &  $f$  \\ 
$b$ &  $b$ & $b$ & $b$  & $b$ &  $b$  \\
$e$ &  $a$ & $f$ & $b$  & $e$ &  $e$ \\
$1$ &  $a$ & $f$ & $b$  & $e$ &  $1$ \\
\end{tabular}, \text{ } with order\text{ }  \begin{tabular}{| r c l |  }
\hline	
& $b$   &   \\
$/$ & $\mid$  & $\setminus$ \\ 
$1$ & $e$ & $f$\\
$\setminus$ & $\mid$  & $/$ \\ 
& $a$   &   \\
\hline  
\end{tabular} \\
\end{center}

then it is easy to check that $U$ is right pounitary in $S$ but since $1\le be$ and $ae\le 1$, $U$ is neither upper or lower strongly right pounitary in $S$.

\medskip

Notice that if $U$ is strongly right pounitary in $S$ then $S\setminus U$ is a right $U-$poset and $S$ is the coproduct in the category of right $U-$posets of $U$ and $S\setminus U$. In other words, within the category right $U-$posets, $U$ is a direct summand of $S$. This is exactly the situation in the unordered context (see for example~\cite{renshaw-2002}), and so for this reason we feel that the strongly right pounitary property is a very natural analogue for the right unitary property within the category of posets over pomonoids.

\bigskip

Let $X$ be a left $U-$poset and suppose that $1 \otimes x \le 1 \otimes x'$ in $S\otimes_U X$. Then there exists $n\ge 1$ and $x_2, \dots, x_n \in X, s_1, \dots, s_n \in S$, and $u_1, \dots, u_n, v_1, \dots, v_n \in U$ such that
\begin{align*}
1&  \le  s_1u_1&u_1x& \le v_1x_2 \\
s_1v_1 & \le s_2u_2 &  u_2x_2 &  \le v_2x_3\\
\vdots &  & \vdots\\
s_{n-1}v_{n-1} & \le s_nu_n & u_nx_n&  \le v_nx'\\
s_nv_n & \le 1.
\end{align*}

If $U$ is right pounitary in $S$ then $s_i\in U$ for $1\le i\le n$ and hence $x\le s_1u_1x\le s_1v_1x_2\le\ldots\le x'$ and so $U$ has the left extension property in $S$. Consequently we can deduce

\begin{theorem} \label{p2}
Let $U$ be a (left, right) pounitary subpomonoid of a pomonoid $S$. Then $U$ has the (right, left)  poextension property  in $S$. 
\end{theorem}

Let $f: X \to Y$ be a $U-$order embedding. Then, we can extend the pounitary concepts as follows.
Let  $u, v, u_1, u'_1, \dots  u_n, u'_n \in U, y, y_1, y_2, \dots y_n \in Y, x\in X$. Then we say that

\begin{enumerate}
\item $f$ is {\em upper strongly right pounitary} if $f(x) \le y u \Rightarrow y \in \text{im }f$;
\item $f$ is {\em lower strongly right pounitary} if $y u \le f(x) \Rightarrow y \in \text{im }f$;
\item $f$ is {\em strongly right pounitary} if ($ f(x) \le y u$  or $y u \le f(x)) \Rightarrow y \in \text{im }f$;
\item $f$ is {\em right pounitary} if $f(x) \le y_1 u_1, y_1 u'_1 \le y_2 u_2, \dots y_n u'_n \le f(x') \Rightarrow  y_1,  y_2, \dots,  y_n \in \text{im }f$;
\item $f$ is {\em right unitary} if $ y u = f(x) \Rightarrow y \in \text{im }f$.
\end{enumerate}

The left-handed versions of these properties can be defined dually. Notice that if $f:X\to Y$ is right (left) pounitary and if $f(x)\le y\le f(x')$ with $x,x'\in X, y\in Y$ then clearly $y\in\im(f)$ and so $f$ is convex.

\medskip

The following statement is easy to prove and will be used later.

\begin{lemma} \label{r2}
Let $f: X \to Y$ be a lower strongly right pounitary $U-$poset morphism and $A$ a left $U-$poset. Then if $y \otimes a \le f(x) \otimes a'$ in $Y\otimes_U A$ then $y\in\im(f)$.
\end{lemma}

\begin{proof}
Suppose that $y\otimes a\le f(x)\otimes a'$ in $Y\otimes_U A$. Then there exists $n\ge 1$ and $y_1,\ldots, y_n\in Y, a_2,\ldots, a_n\in A$ and $u_1,\ldots, u_n, v_1, \ldots, v_n\in U$ such that

\begin{align*}
y& \le y_{1}u_{1} &   u_{1} a&  \le v_{1} a_{2}\\
y_{1}v_{1} & \le y_{2}u_{2} &  u_{2}a_{2} &  \le v_{2} a_{3}\\
\vdots &  & \vdots\\
y_{n-1}v_{n-1} & \le y_nu_n & u_{n}a_{n}&  \le v_{n} a'\\
y_{n}v_{n} & \le f(x).
\end{align*}
Since $f$ is lower strongly right unitary, $y_n\in \im(f)$. Consequently $y_{n-1}\in\im(f)$ and continuing in this fashion we see that $y\in \im(f)$.
\end{proof}

\begin{lemma} \label{q2}
 Let $U$ be a submonoid of a pomonoid $S$ and let $f: X \to Y$  be a right pounitary $U-$poset morphism. Then the induced map $f\otimes1:X\otimes_US\to Y\otimes_US$ is an order embedding.
\end{lemma}

\begin{proof}
Suppose that $f(x)\otimes s\le f(x')\otimes s'$ in $Y\otimes_U S$ so that we have a scheme
\begin{align*}
f(x)&  \le  y_1u_1&u_1s& \le v_1s_2 \\
y_1v_1 & \le y_2u_2 &  u_2s_2 &  \le v_2s_3\\
\vdots &  & \vdots\\
y_{n-1}v_{n-1} & \le y_nu_n & u_ns_n&  \le v_ns'\\
y_nv_n & \le f(x').
\end{align*}
Since $f$ is pounitary, there exists $x_i\in X, 1\le i\le n$ such that $y_i=f(x_i)$ and since $f$ is right pounitary and hence an order embedding, we have a scheme
\begin{align*}
x&  \le  x_1u_1&u_1s& \le v_1s_2 \\
x_1v_1 & \le x_2u_2 &  u_2s_2 &  \le v_2s_3\\
\vdots &  & \vdots\\
x_{n-1}v_{n-1} & \le x_nu_n & u_ns_n&  \le v_ns'\\
x_nv_n & \le x'
\end{align*}
and so $x\otimes s\le x'\otimes s'$ in $X\otimes_US$.
\end{proof}

That many of these unitary properties are preserved under direct limits is demonstrated by the following result.

\begin{lemma} \label{hj17}
Let $(X_i, \varphi^i_j)$ be directed system of $U-$posets with direct limit $(X, \varphi_i)$. Then, $\varphi_i$ is right (upper, lower) strongly  pounitary  if and only if $\varphi^i_j$ is right (upper, lower) strongly  pounitary, where $i \le j$.
\end{lemma}

\begin{proof}
We will prove this result for upper strongly pounitary, the case for lower strongly pounitary is similar.   Suppose that $\varphi^i_j$ is right upper strongly  pounitary and $ \varphi_i(x_i) \le \varphi_j(x_j)u$. Then, from Lemma~\ref{direct-limit-lemma} there exists $k \ge i, j$ such that $\varphi_k^i(x_i)  \le \varphi_k^j(x_j)u$. Since $\varphi_k^i$ is right upper strongly  pounitary, there exists $x'_i \in X_i$ such that $\varphi_k^i(x'_i) = \varphi_k^j(x_j)$. Hence, again from Lemma~\ref{direct-limit-lemma} $\varphi_i(x'_i) = \varphi_j(x_j)$. Therefore, $\varphi_i$ is right upper strongly  pounitary.

\smallskip

Conversely, suppose that $\varphi_i$ is right upper strongly  pounitary, and suppose that $\varphi_j^i(x_i)$ $\le$ $ x_j s$. Then, $\varphi_i(x_i) = \varphi_j \varphi_j^i(x_i)\le\varphi_j(x_j u) = \varphi_j(x_j) u$. Since $\varphi_i$ is a upper right strongly  pounitary, there exists $x'_i \in X_i$ such that $\varphi_j(x_j) = \varphi_i(x'_i) = \varphi_j \varphi_j^i(x'_i)$. Hence, $x_j = \varphi_j^i(x'_i)$, since $\varphi_j$ is also an order embedding. This completes the proof.
\end{proof}

\begin{theorem} \label{w2}
Let $U$ be a strongly pounitary subpomonoid of a pomonoid $S$. Then for every $(U,S)-$poset $X$ and every $(U, U)-$poset $Y$ and every strongly pounitary $f: X \to Y$  there exist a $(U,S)-$poset order embedding  $h: X \to F(S;X,Y)$ and a $(U, U)-$strongly  pounitary order embedding $g: Y \to F(S;X,Y)$ such that $g \circ f = h$.
\end{theorem}

\begin{proof} From Theorem \ref{p2},  $U$ has the poextension property in $S$ and  from Lemma \ref{q2}, $f\otimes1:X\otimes_US\to Y\otimes_US$ is an order embedding.
Let the maps $g: Y \to F(S; X, Y)$ and $h: X \to F(S; X, Y)$ be defined  as in Theorem \ref{b1} and suppose that $g(y) \le  g(y')$. Then,  $(y \otimes 1) \sigma \le (y' \otimes 1) \sigma$ and so  $y \otimes 1 \le_{\alpha(R)} y' \otimes 1$. Hence there exists $n\ge 0, s_{i}, s'_{i} \in S$, and $y_{i}, y'_{i} \in Y, 0\le i\le n$ such that
$$
y \otimes 1 \le y_{1} \otimes s_{1} \alpha(R)  y'_{1} \otimes s'_{1} \le  \dots \le y_{n} \otimes s_{n} \alpha(R)  y'_{n} \otimes s'_{n} \le y' \otimes 1.
$$
We can assume that the number of $\alpha(R)$ terms is minimal. If there are no such terms then $y\otimes1\le y'\otimes1$ in $Y\otimes_US$ and so $y\le y'$ since $U$ has the poextension property in $S$. Otherwise for each  $i$ there exists a scheme
\begin{align*}
y_{i} \otimes s_{i}& = y_{i1} \otimes s_{i1} t_{1}  \\
y'_{i1} \otimes s'_{i1} t_{1}&  = y_{i2} \otimes s_{i2} t_{2} \\
&\vdots\tag{$\ast\ast$}\\
y'_{im} \otimes s'_{im} t_{m} & = y'_{i} \otimes s'_{i}
\end{align*}
where $s_{ij}, s'_{ij} \in S$,  $y_{ij}, y'_{ij} \in Y$, $t_{j} \in U$ and ($y_{ij} \otimes s_{ij}$, $y'_{ij} \otimes s'_{ij}$) $\in R$. From the definition of $R$, $y_{ij} \otimes s_{ij}$ = $f(x_{ij}) \otimes r_{ij}$, $y'_{ij} \otimes s'_{ij}$ = $f(x'_{ij}) \otimes r'_{ij}$, and $x_{ij}r_{ij} \le x'_{ij}r'_{ij}$. Hence $x_{ij}r_{ij}t_{j} \le x'_{ij}r'_{ij}t_{j}$ and so $x'_{ij} \otimes r'_{ij}t_{j} = x_{i(j+1)} \otimes r_{i(j+1)}t_{j+1}$ since $f\otimes1$ is an order embedding. Because the canonical map $X\otimes_US\to X$ is monotone, we deduce
$$
x_{i1} r_{i1} t_1 \le  x'_{i1} r'_{i1} t'_i = x_{i2} s_{i2} t_2\le  \dots \le  x_{im} r_{im} t_m.
$$

Since $f$ is  strongly pounitary, we have $y_i = f(x_i)$ and $y'_i = f(x'_i)$ for some $x_i, x'_i\in X$ and hence
$$
x_i s_i \le  x_{i1} r_{i1} t_1 \le  x'_{i1} r'_{i1} t'_i = x_{i2} r_{i2} t_2\le  \dots \le  x_{im} r_{im} t_m \le x'_i s'_i.
$$ 

By Lemma~\ref{r2} and its dual it follows that there exists $x,x'\in X$ such that $ y=f(x), y'=f(x')$. Consequently $y \otimes 1 =f(x)\otimes1\le f(x_1)\otimes s_1 = y_1 \otimes s_1$, and since $f\otimes1$ is an order embedding that $x\le x_1s_1$. In a similar way $x'_ns'_n\le x'$ and so $ x \le x_1 s_{1} \le  x'_{1} s'_{1}  \le \dots \le x'_n s'_n \le x'$ from which it follows that $g$ is an order embedding. 

\smallskip

We next show that $g$ is upper strongly pounitary (that $g$ is lower strongly pounitary follows from a similar argument). Suppose that  $(y \otimes 1) \sigma \le (y' \otimes s) \sigma u = (y' \otimes su) \sigma $ so that $y \otimes 1 \le_{\alpha(R)} y' \otimes su$. Hence there exists $n\ge 1$ and $s_{i}, s'_{i} \in S$, and $y_{i}, y'_{i} \in Y, 1\le i\le n$ such that
$$
y \otimes 1 \le y_{1} \otimes s_{1} \alpha(R)  y'_{1} \otimes s'_{1} \le \dots \le y_{n} \otimes s_{n} \alpha(R)  y'_{n} \otimes s'_{n} \le y' \otimes su.
$$
As before, assume that the number of $\alpha(R)$ terms is minimal. If there are no such terms then we deduce that
$ y \otimes 1 \le y' \otimes su$. By the dual of Lemma~\ref{r2} it follows that $su\in U$ and so $s\in U$. Hence, $(y' \otimes s) \sigma = (y's \otimes1) \sigma = g(y's)$.
Then, again using Lemma~\ref{r2} and its dual we deduce that there exists $x,x'\in X$ with $y = f(x), y' = f(x')$. Consequently
$(y'\otimes s)\sigma = (f(x')\otimes s)\sigma = (f(x's)\otimes1)\sigma = g(f(x's))$ and this shows that $g$ is upper strongly pounitary as required. Notice that $h=g\circ f$ is a $(U,S)-$poset map by Theorem \ref{b1} and is clearly an order-embedding.
\end{proof}

In the above result, the left $U-$poset structure appears to play no role, but it is intrinsically used in the main result of this section, which we are now in a position to state and prove.

\begin{theorem} \label{mo1}
Let $[U; S_1, S_2]$ be a pomonoid  amalgam. If  $U$ is strongly pounitary in both $S_1$ and $S_2$ then the amalgam is strongly poembeddable and $U$ has the poextension property in $S_1 \ast_U S_2$.
\end{theorem}

\begin{proof}
Construct the direct system $(Y_n, k_n)$ as in Theorem~\ref{direct-limit-theorem}. From Theorem \ref{p2}, $U$ has the poextension property in both $S_1$ and $S_2$ and hence, the map $S_1=Y_1 \to Y_2=S_1 \otimes S_2$ is an order embedding.  We next show that the map $k_1: Y_1 \to Y_2$ is strongly pounitary. Suppose that $k_1(s_1) \le y_2 u$ so that $s_1 \otimes 1 \le  s'_1 \otimes s_2u$ for some $s_1,s'_1\in S_1, s_2\in S_2, u \in U$. Since $U$ is strongly pounitary, and hence lower strongly pounitary, in $S_2$, by the dual of Lemma~\ref{r2} we can deduce that $s_2u\in \im(U\to S_2) = U$ and so $s_2 \in U$. Hence $y_2 =   s'_1 \otimes s_2 = s'_1 s_2 \otimes 1 = k_1(s'_1 s_2)$ as required. In a similar way the map $S_2\to Y_2$ is strongly pounitary. Since $Y_1=S_1$, it follows from Theorem~\ref{w2} that $Y_2\to Y_3$ is a strongly pounitary order embedding and so an inductive argument then allows us to deduce that for all $n\ge 1$, the map $k_n: Y_n \to Y_{n+1}$ is a strongly pounitary order embedding. Hence,  the amalgam is weakly poembeddable by Lemma~\ref{poamalgam-weak-embeddability-condition-lemma}.

To show that the amalgam is strongly poembeddable we use Lemma~\ref{poamalgam-strong-embeddability-condition-lemma}. Suppose therefore that $s_1 \otimes 1 = 1 \otimes s_2$ in $S_1 \otimes S_2$. Then there exists a scheme

\begin{align*}
s_1 & \le s_{11}u_{1} &    u_1 &  \le v_{1} s_{21}\\
s_{11}v_{1} & \le s_{12}u_{2} &  u_{2}s_{21} &  \le v_{2} s_{22}\\
\vdots & & \vdots \\
s_{1 m-1}v_{m-1} & \le  s_{1 m}u_m & u_{m}s_{2 m-1}& \le v_ms_2\\
s_{1 m}v_m & \le  1. &
\end{align*}

Since $U$ is  strongly pounitary in $S_1$, $s_1, s_{11}, \dots, s_{1 m} \in U$. Hence, $s_1 \otimes 1 \le 1 \otimes s_2$ in $U \otimes S_2\cong S_2$ and so $s_1\le s_2$. In a similar way, $s_2\le s_1$ and so $s_1 = s_2 \in U$ and the amalgam is strongly poembeddable.

It follows from Theorem~\ref{w2}, Theorem~\ref{direct-limit-theorem} and Lemma~\ref{hj17} that the order embedding $S_1\to S_1\ast_U S_2$ is strongly pounitary and since the composite of strongly pounitary order embeddings is also strongly pounitary, $U$ is strongly pounitary in $S_1\ast_US_2$.
\end{proof}

\section{Commutative pomonoid amalgams}

\begin{lemma} \label{tr}
Let $U$ be a subpomonoid  of a pomonoid $S$ and suppose that $U$ has the poextension (extension) property in $S$. If $\lambda: X \to Y$ is a convex $U-$poset morphism and if $y \otimes 1 = \lambda(x) \otimes s $ in $Y \otimes S$, then $y \in \im\lambda$.
\end {lemma}

\begin{proof}
Suppose that $y \otimes 1 = \lambda(x) \otimes s $ in $Y \otimes_U S$ and consider the commutative diagram
$$
\begin{diagram}
\node{X} \arrow{e,t}{\lambda} \arrow{s,l}{\lambda} \node{Y} \arrow{s,r}{\alpha}  \\
\node{Y} \arrow{e,b}  {\beta} \node{P} 
\end{diagram}
$$

where $P$ is the pushout. It is known from \cite{bf-2005} that tensor products preserve pushouts in the category of $U-$posets so in $P \otimes_U S$ we have $\alpha(y) \otimes 1 =\alpha\lambda(x)\otimes s = \beta\lambda(x)\otimes s = \beta(y) \otimes 1$. Since $U$ has the extension property in $S$, $\alpha(y) =\beta(y) $ and hence from Lemma \ref{bad} there exist $x \in X$ such that $y = \lambda(x)$.
\end{proof}

From now on we assume that $S$ is a commutative pomonoid. It is well-known (see for example~\cite{sohail-2011}) that in the category of commutative pomonoids the amalgamated free product of a pomonoid amalgam $[U;S_1,S_2]$ reduces to the tensor product $S_1\otimes_US_2$ with multiplication given by $(s_1\otimes s_2)(s'_1\otimes s'_2) = s_1s'_1\otimes s_2s'_2$.

\begin{theorem} \label{p1}
Let $[U, S_1, S_2]$ be commutative pomonoid amalgam and $U$ has the poextension (resp. extension) property and is convex in both $S_1$ and $S_2$. Then the amalgam is strongly poembeddable (resp. embeddable).
\end{theorem}

\begin{proof}
Since $U$ has the poextension (resp. extension) property in $S_1$ and $S_2$, $\lambda_1: S_1 \to S_1 \otimes S_2$ and $\lambda_1: S_2 \to S_1 \otimes S_2$ are order embeddings (resp. monomorphisms) and so the commutative amalgam is weakly poembeddable (resp. embeddable) in its amalgamated free product. Now, suppose that $\lambda_1(s_1) = \lambda_2(s_2)$. It follows that $s_1\otimes1=1\otimes s_2$ in $S_1\otimes_U S_2$ and so from Lemma~\ref{tr} we can deduce that $s_1\in U$ and (dually) $s_2\in U$, and consequently $s_1=s_2 \in U$. Hence the amalgam is strongly poembedable (resp. embeddable).
\end{proof}

Since pounitary morphisms are convex, as a consequence we have

\begin{corollary}
Let $[U, S_1, S_2]$ be commutative pomonoid amalgam and $U$ be pounitary in $S_1$ and $S_2$. Then the amalgam is strongly poembeddable.
\end{corollary}

A pomonoid $S$ is said to be {\em left pocancellative} if $x\le y$ whenever $sx\le sy$. {\em Right pocancellativity} is defined dually and $S$ is {\em pocancellative} if it is right and left pocancellative.

\begin{theorem}
Let $[U; S_1, S_2]$ be  pocancellative commutative pomonoid amalgam and let $U$ be abelian pogroup. Then the amalgam is poembeddable in the class of pocancellative commutative pomonoids.
\end{theorem}

\begin{proof}
It is known from  \cite[Theorem 3]{bf-2011} that pogroups are  poamalgamation bases in the class of pomonoids. Consequently $[U; S_1, S_2]$ is weakly poembeddable in its amalgamated free product, $S_1\otimes_U S_2$. If $s_1\otimes1=1\otimes s_2$ in $S_1\otimes_U S_2$ then applying the map $\varphi_2$ in Theorem~\ref{direct-limit-theorem} and using the strong embeddability of $[U;S_1,S_2]$ in $S_1\ast_U S_2$ in the category of pomonoids we can easily deduce that $s_1=s_2\in U$ and the amalgam is strongly posembeddable in $S_1\otimes_U S_2$.

To prove that $S_1 \otimes S_2$ is pocancellative suppose that $(s_1 \otimes  s_2)$ $(t_1 \otimes  t_2)$ $\le$ $(s'_1 \otimes  s'_2)$ $(t_1 \otimes t_2)$ so that $(s_1 t_1 \otimes s_2 t_2)$  $\le$ $(s'_1 t_1 \otimes s'_2 t_2)$. Hence, there exists a scheme of inequality such that
\begin{align*}
s_1 t_1& \le x_{1}u_{1} &  u_{1} s_2 t_2&  \le v_{1} y_{2}\\
x_{1}v_{1} & \le x_{2}u_{2} &  u_{2}y_{2} &  \le v_{2} y_{3}\\
\vdots &  & \vdots\\
x_{n-1}v_{n-1} & \le x_nu_n& u_{n}y_{n}&  \le v_{n} s'_2 t_2\\
x_{n}v_{n} & \le s'_1 t_1 &
\end{align*}
where $x_1, \dots, x_n \in S_1$, $y_2, \dots, y_n \in S_2$, and $u_1, \dots, u_n, v_1, \dots, v_n \in U$. 
Hence $s_1 t_1 p$ $\le$ $s'_1 t_1$ and $p^{-1} s_2 t_2$ $\le$ $s'_2 t_2$, where $p = u_1^{-1} v_1  u_2^{-1} v_2 \dots  u_n^{-1} v_n$ . Since $S_i$ is pocancellative and commutative, then $s_1  p$ $\le$ $s'_1$ and $p^{-1} s_2$ $\le$ $s'_2$ from which we can easily deduce that $s_1 \otimes s_2$ $\le$ $s'_1 \otimes  s'_2$.
\end{proof} 

\begin{lemma}
Let $S$ be a commutative pocancellative pomonoid. Then $S$ is poembeddable in a pogroup.
\end{lemma}

\begin{proof}
Consider the pomonoid  $S \times S$ with partial order relation:
$$
(s, t) \le (s', t') \hbox{\rm\  if and only if } s \le s' \hbox{\rm\ and }t' \le t.
$$ Define $G(S) = (S \times S)/ \Delta$, where  
$$
\Delta = \lbrace ((s_1, t_1), (s_2, t_2)): s_1 t_2 = s_2 t_1 \rbrace.
$$
Clearly $\Delta$ is monoid congruence. Denote a typical element of $G(S)$ by $[(s, t)]$ and note that for all $s,t \in S, [(s,s)]=[(t,t)]=[(1,1)]$ is the identity of $G(S)$.  It is possible to endow the commutative monoid $G(S)$ with the following order
$$
[(s, t)]\le[(s', t')]\text{ if and only if }st' \le s't.
$$
It is straightforward to prove that this is a compatible partial ordered relation and it is clear that the map $S \times S \rightarrow G(S)$ is then monotone. Hence $\Delta$ is a pomonoid congruence. 
For any $s,t,p,q, \in S$ it is clear that $[(s,t)] = [(sq,tp)][(p,q)]$ and so $G(S)$ is an abelian pogroup.
Define a monoid morphism $\chi: S \rightarrow G(S)$ by $\chi(s) = [(s, 1)]$. Then it is obvious that  $s \le t$ if and only if $[(s, 1)] \le [(t, 1)]$. Consequently, $\chi$ is an order embedding as required.
\end{proof}

\begin{theorem}
Any commutative pocancellative pomonoid amalgam is weakly poembeddable in a commutative pomonoid.
\end{theorem}

\begin{proof}
Suppose $[U; S_1, S_2;\varphi_1,\varphi_2]$ is a commutative pocancellative pomonoid amalgam. Consider the amalgam of pogroups $[G(U); G(S_1), G(S_2);\varphi'_i,\varphi'_2]$ where $\varphi'_i: G(U) \rightarrow G(S_i)$ is given by $\varphi'_i([(u,v)] \mapsto [(\varphi_i(u), \varphi_i(v))]$. It is easy to check that $\varphi'_i$ is an order embedding. By \cite[Corollary 2]{sohail-2011}, $[G(U); G(S_1), G(S_2)]$ is poembeddable into a pogroup. Hence, $[U; S_1, S_2]$ is also weakly poembeddable in the pomonoid $S_1\ast_US_2$ and so by Lemma~\ref{poamalgam-weak-embeddability-condition-lemma} in $S_1\otimes_US_2$, which is a commutative pomonoid.
\end{proof}

It is not clear that the amalgam can be strongly poembedded into a commutative pomonoid.

\bigskip

The authors would like to thank the anonymous referee for many helpful comments and suggestions which greatly improved the exposition of this work, and for bringing to our attention the work done in~\cite{sohail-2014}.

\end{document}